\documentclass[11pt]{amsart}
\usepackage[german, english]{babel}
\usepackage{geometry}                % See geometry.pdf to learn the layout options. There are lots.
\usepackage{graphicx}
\usepackage{amssymb}
\usepackage{epstopdf}
\usepackage{tikz}
\usetikzlibrary{knots}
\usepackage{braids}
\usepackage{calc}
\usepackage{color}

\usepackage{array}

\newcommand{\of}[2]{\ensuremath{#1\!\left(#2\right)}}%function of 
\newcommand{\subof}[3]{\of{#1_{#2}}{#3}}%subscripted function of

\newcommand{\bmult}{\ensuremath{\circledast}}
\newcommand{\bprod}[2]{\ensuremath{#1\bmult #2}}
\newcommand{\invbar}[1]{\ensuremath{\bar{#1}}}% inverse of group element (bar notation)
\newcommand{\ldiv}[2]{\ensuremath{#1\backslash #2}}%#2 left divided by #1
\newcommand{\mir}[1]{\ensuremath{#1^{\star}}}
\newcommand{\pr}{\ensuremath{^{\prime}}}
\newcommand{\setbld}[2]{\ensuremath{\left\{#1\,\vert\, #2\right\}}}% ``set-builder'' notation
\newcommand{\single}[1]{\ensuremath{ \left\{ #1 \right\} }}%enclose in braces
\newcommand{\nlist}[1]{\ensuremath{\single{1,2,\dots,#1}}}
\renewcommand{\sharp}{\ensuremath{\#}}%number

\newcommand{\ap}{\ensuremath{a\pr}}
\newcommand{\as}[1]{\ensuremath{a_{#1}}}%subscripted a
\newcommand{\As}[1]{\ensuremath{A_{#1}}}%subscripted capital A
\newcommand{\bp}{\ensuremath{b\pr}}
\newcommand{\bpi}{\ensuremath{\invbar{\pi}}}
\newcommand{\bpis}[1]{\ensuremath{{\bpi}_{#1}}}
\newcommand{\Br}[1]{\ensuremath{\mathfrak{B}_{#1}}}%subscripted fraktur B
\newcommand{\Brp}[1]{\ensuremath{\mathfrak{B}_{#1}^{+}}}%subscripted fraktur B plus (positive braids)
\newcommand{\bs}[1]{\ensuremath{b_{#1}}}%subscripted b
\newcommand{\Bs}[1]{\ensuremath{B_{#1}}}%subscripted capital B

\newcommand{\Cross}[1]{\ensuremath{\mathcal{C}(#1)}}

\newcommand{\deffont}[1]{\emph{\textbf{#1}}}
\newcommand{\eqdef}{\ensuremath{:=}}% equal by definition

\newcommand{\Gs}[1]{\ensuremath{G_{#1}}}

\newcommand{\ks}[1]{\ensuremath{k_{#1}}}
\newcommand{\Ks}[1]{\ensuremath{K_{#1}}}
\newcommand{\map}[3]{\mbox{\ensuremath{#1\!\colon\!%
	#2\negthinspace\rightarrow\negthinspace#3}}}
\newcommand{\OR}[1]{\ensuremath{\of{OR}{#1}}}
\newcommand{\pbraid}[1]{\ensuremath{#1^{+}}}
\newcommand{\pis}[1]{\ensuremath{\pi_{#1}}}
\newcommand{\piof}[1]{\of{\pi}{#1}}
\newcommand{\pisof}[2]{\ensuremath{\of{\pis{#1}}{#2}}}
\newcommand{\ps}[1]{\ensuremath{p_{#1}}}
\newcommand{\psof}[2]{\ensuremath{\subof{p}{#1}{#2}}}
\newcommand{\Pur}[1]{\ensuremath{\mathfrak{P}_{#1}}}%subscripted fraktur P
\newcommand{\Purp}[1]{\ensuremath{\mathfrak{P}^{+}_{#1}}}%subscripted fraktur P plus (pure positive braids)
\newcommand{\qs}[1]{\ensuremath{q_{#1}}}
\newcommand{\resp}[1]{(\textit{resp}. #1)}

\newcommand{\rs}[1]{\ensuremath{r_{#1}}}
\newcommand{\Rs}[1]{\ensuremath{R_{#1}}}
\newcommand{\sigof}[1]{\ensuremath{\sigma\left(#1\right)}}
\newcommand{\SR}[1]{\ensuremath{\mathcal{SR}_{#1}}}%calligraphic SR (realizable matrices) 
\newcommand{\SRp}[1]{\ensuremath{\mathcal{SR}^{+}_{#1}}}%calligraphic SR plus (T0 matrices with positive SR decomp)
\newcommand{\SRT}[1]{\of{SRT}{#1}}
\newcommand{\ssub}[1]{\ensuremath{s_{#1}}}%subscripted s
\newcommand{\Ssub}[1]{\ensuremath{S_{#1}}}%subscripted capital S
\newcommand{\SymMat}[1]{\ensuremath{\mathfrak{S}_{#1}^{0}[\mathbf{Z}]}}%(symmetric integer matrices with zero diagonal)
\newcommand{\SymMatp}[1]{\ensuremath{\mathfrak{S}_{#1}^{0}\lbrack\mathbf{Z}^{+}\rbrack}}%
\newcommand{\SymGp}[1]{\ensuremath{\Sigma_{#1}}}
\newcommand{\taus}[1]{\ensuremath{\tau_{#1}}}

\newcommand{\ts}[1]{\ensuremath{t_{#1}}}

\newcommand{\Vs}[1]{\ensuremath{V_{#1}}}
\newcommand{\xs}[1]{\ensuremath{x_{#1}}}

\newcommand{\subind}{\ensuremath{\mathcal{I}}}

\newcommand{\AI}{\ensuremath{\As{\subind}}}

\newcommand{\piI}{\ensuremath{\pis{\subind}}}

\newcommand{\sqsize}[1]{\ensuremath{#1\times #1}}

\newtheorem{theorem}{Theorem}
\newtheorem{lemma}[theorem]{Lemma}%
\newtheorem{definition}[theorem]{Definition}%
\newtheorem{remark}[theorem]{Remark}%
\newtheorem{prop}[theorem]{Proposition}%
\newtheorem{conjecture}[theorem]{Conjecture}

\newcommand{\labbox}[1]{\colorbox[gray]{.80}{#1}}

\newcommand{\fourTableau}[6]{
	\begin{array}{cccc}
		\hline
		\multicolumn{1}{|c}{\labbox{1}} & 
		\multicolumn{1}{|c}{#1} & 
		\multicolumn{1}{|c}{#2} & 
		\multicolumn{1}{|c|}{#3} 
		\\\hline 
	& 
		\multicolumn{1}{|c}{\labbox{2}} & 
		\multicolumn{1}{|c}{#4} & 
		\multicolumn{1}{|c|}{#5} 
		\\\cline{2-4}
	& & 
		\multicolumn{1}{|c}{\labbox{3}} & 
		\multicolumn{1}{|c|}{#6} 
		\\\cline{3-4}
	& & & 
		\multicolumn{1}{|c|}{\labbox{4}} 
		\\\cline{4-4}
	\end{array}%
	}%SR Tableau of size 4

\newcommand{\threeConfig}[6]{
	\begin{array}{ccc}
		\hline
		\multicolumn{1}{|c}{\labbox{#1}} & 
		\multicolumn{1}{|c}{#4} & 
		\multicolumn{1}{|c|}{#5} 
		\\\hline 
	& 
		\multicolumn{1}{|c}{\labbox{#2}} & 
		\multicolumn{1}{|c|}{#6} 
		\\\cline{2-3}
	&&
		\multicolumn{1}{|c|}{\labbox{#3}} 
		\\\cline{3-3}
	\end{array}%
	}%SR Configuration Tableau of size 3; Usage:threeConfig{i}{k}{j}{x}{y}{z}

\begin{document}

\title{Crossing matrices of positive braids}
\author{Mauricio Gutierrez}
	\address{}
	\email{magutierrez002@gmail.com}
\author{ Zbigniew Nitecki}
	\address{Department of Mathematics\\Tufts University\\503 Boston Avenue\\Medford, MA 02155}
	\email{znitecki@tufts.edu}
\date{\today}                                           % Activate to display a given date or no date
\begin{abstract}
	The crossing matrix of a braid on $N$ strands is the $N\times N$ integer matrix with zero diagonal 
	whose $i,j$ entry is the algebraic number (positive minus negative) of crossings by strand $i$ over strand $j$ . 
	When restricted to the subgroup of pure braids, this defines a homomorphism onto the additive subgroup 
	of $N\times N$ symmetric integer matrices with zero diagonal--in fact, it represents the abelianization of this subgroup.
	As a function on the whole $N$-braid group, it is a derivation defined by the action of the symmetric group on square 
	matrices.  The set of all crossing matrices can be described using the natural decomposition of any braid as the product
	of a pure braid with a ``permutation braid'' in the sense of Thurston, but the subset of crossing matrices for positive
	braids is harder to describe.  We formulate a finite algorithm which exhibits all positive braids with a given 
	crossing matrix, if any exist, or declares that there are none.
\end{abstract}



\maketitle

\section{Introduction}
The notion of a crossing matrix was formulated in \cite{BGKN}: for a geometric braid $b$ with $N$ strands, the crossing matrix
\Cross{b} is an \sqsize{N} matrix whose $i,j$ entry is the (algebraic) number of crossings of strand $i$ over strand $j$.
This matrix is invariant under the braid relations on geometric braids and hence is well-defined as a function on the braid 
group  \Br{N}. In section~\ref{sec:CM} we review the basic properties of this function 
and the relatively easy characterization of its range.  
The current paper focuses on the deeper problem of characterizing those matrices that arise as crossing matrices 
of \emph{positive} braids--that is, geometric braids whose crossings all go in the same direction (left over right).  
The reason for asking for such a characterization is that in this case there is no cancellation of a left-over-right crossing with a
right-over-left crossing by the same strands--all crossings ``show up'' in the matrix. 
Examples show that a crossing matrix with all entries non-negative may nonetheless not represent any \emph{positive} braid. 
In this paper, we explain several obstructions to being the crossing matrix of a positive braid, and give an algorithm 
which displays all the positive braids with a given matrix as their crossing matrix (including deciding when there are none).%
\footnote{This algorithm has been successfully implemented in a Mathematica program, which can handle examples up to about size
\sqsize{7} on a MacBook Pro with 8GB of memory.}
We have not been able to come up with a conceptual characterization of which matrices arise as crossing matrices of 
positive braids in general, although we conjecture a characterization for \emph{pure} positive braids, resting on a very specific 
(conjectured) lemma.

\section{Crossing Matrices of Braids}\label{sec:CM}

\subsection{Definition of crossing matrices}\label{subsec:def}  
We think of a \deffont{braid} with $N$ strands as the homotopy class, modulo endpoints, 
of a \deffont{geometric braid}: an ensemble of $N$ 
differentiable paths \psof{i}{t}, $i=1,\dots,N$ in the plane (the \deffont{strands} of the braid) 
whose tangent is never horizontal; the set of (distinct) initial positions 
$\{\psof{i}{0}\ :\ i=1,\dots,N\}$
and the set of final positions $\{\psof{i}{1}\ :\ i=1,\dots,N\}$ differ only in their common vertical coordinate, which we take to be $1$ for the initial
and $0$ for the final points;  it is not required that the initial and final positions of any particular strand be horizontally aligned.
We assume for convenience that these paths are pairwise transverse, so that there are only finitely many crossings between strands,
and each such crossing is assigned \emph{crossing data:} it is regarded as \emph{positive} \resp{\emph{negative}} 
depending on whether it is left-over-right or right-over-left as we move down the path.%(See Figure~\ref{fig:gbraid})
%\begin{figure}
%\begin{center}
%\begin{tikzpicture}
%	\braid[
%	line width=2pt,
%	style strands={1}{red},
%	style strands={2}{blue},
%	style strands={3}{green},
%	style strands={4}{brown},
%	style strands={5}{black}
%	]
%	(braid) at (2,0) s_{1} s_{4}  s_{2}^{-1}  s_{2}^{-1} s_{4}^{-1};
%%	
%	\node[at=(braid-1-s),pin=north : 1] {};
%	\node[at=(braid-2-s),pin=north : 2] {};
%	\node[at=(braid-3-s),pin=north : 3] {};
%	\node[at=(braid-4-s),pin=north : 4] {};
%	\node[at=(braid-5-s),pin=north : 5] {};
%
%	\node[at=(braid-1-e),pin=south : 1] {};
%	\node[at=(braid-2-e),pin=south : 2] {};
%	\node[at=(braid-3-e),pin=south : 3] {};
%	\node[at=(braid-4-e),pin=south : 4] {};
%	\node[at=(braid-5-e),pin=south : 5] {};
%
%\end{tikzpicture}
%\caption{A geometric braid $b$ }
%\label{fig:gbraid}
%\end{center}
%\end{figure} 

Given a geometric braid $b$, we define its \deffont{crossing matrix} as the \sqsize{N} matrix $\boldsymbol{\Cross{b}}$ 
whose $i,j$ entry is the algebraic crossing number (number of positive crossings minus number of negative crossings)
for strand $i$ crossing \emph{over} strand $j$. By definition, the diagonal entries of \Cross{b} are zero.

It is easy to check that the braid relations (which do modify some crossings) don't affect the crossing matrix. So
the crossing matrix \Cross{b} can be thought of as defined on the braid represented by $b$. 
The description of a geometric braid in terms of crossings of \emph{strands} is a point of view used by Thurston \cite{Thurston},
and contrasts with the point of view of Artin (in terms of \emph{generators and relations}) in his original expositions of the braid group 
\Br{N} \cite{Artin1, Artin2}. 

\subsection{The permutation associated to a braid}\label{subsec:perm}
Attached to each geometric braid is the permutation on the set of horizontal coordinates 
of the starting points of strands which takes (the horizontal coordinate of) the starting point of each strand to its ending point.  

We  use Greek letters to denote permutations, regarded as \deffont{rearrangements}%
--that is, permutations act on 
\emph{positions} rather than \emph{elements}%
: a permutation $\pi\in\SymGp{N}$ 
will be specified by the word%
\footnote{
Permutations will act on words on the right, denoted by superscript.  
} %
 $(12\cdots N)^{\pi}=\pis{1}\pis{2}\cdots\pis{N}$ in the numbers $1,2,\dots,N$ resulting from the action of $\pi$ on the word $12\cdots N$.
%\footnote{
This contrasts with regarding a permutation as a bijective mapping \map{\pi}{\nlist{N}}{\nlist{N}} on a set of $N$ elements 
and denoting it  by the $N$-tuple $(\piof{1},\piof{2},\dots,\piof{N})$ of \emph{images} of the individual elements $1,2,\dots,N$ 
under this mapping.
In fact the word $\piof{1}\piof{2}\cdots\piof{N}$ in our ``rearrangement'' notation
denotes the \emph{inverse} permutation%
\footnote{
The inverse of a permutation or braid will be denoted by an overbar: \invbar{\pi}.
}.
For example, the rearrangement $\pi\in\SymGp{4}$ which acts on the positions $1,2,3,4$ via
\begin{equation*}
	\piof{1}=3,\quad\piof{2}=1,\quad\piof{3}=4,\quad\piof{4}=2
\end{equation*}
takes the word $abcd$ to 
\begin{equation*}
	(abcd)^{\pi}=bdac\quad\textit{(e.g., }(1234)^{\pi}=\pis{1}\pis{2}\pis{3}\pis{4}=2413\text{)}
\end{equation*}
while 
\begin{equation*}
	(\piof{1},\piof{2},\piof{3},\piof{4})=(3,1,4,2).
\end{equation*}
%}
%
We extend the ``rearrangement'' action of permutations to matrices: If $A=((\as{ij}))$ is an \sqsize{N} matrix
and $\pi\in\SymGp{N}$, then 
$A^{\pi}$ is the matrix obtained by rearranging the rows as well as the columns of $A$ according to $\pi$.  

When a matrix is the crossing matrix of a braid, say $A=\Cross{a}$, then the permutation \pis{a} associated to $a$ can be read off 
of $A$: every (positive) crossing of strand $i$ \emph{over} another strand moves its position one place to the right, while every (positive)
crossing of another strand over the $i^{th}$ moves it to the left one space.  From this it follows that \pis{a} is defined (\emph{as a mapping
}\map{\pis{a}}{\nlist{N}}{\nlist{N}}) by \emph{adding} to $i$ the sum of the $i^{th}$ \emph{row}  
and \emph{subtracting} the sum of the $i^{th}$ \emph{column} of $A$:%
\footnote{For an arbitrary matrix, this formula does not necessarily define a permutation, only a mapping--which might even map
$\nlist{N}$ to a different set of integers. However, we shall show that in the context we consider it is always a permutation
(Remark~\ref{rmk:perm}).}
\begin{equation}\label{eqn:perm}
	\pisof{A}{i}= i+\sum_{j=1}^{N}A_{ij}-\sum_{k=1}^{N}A_{ki}.
\end{equation}
%(Again, we warn the reader that this is a statement about where \emph{position} $i$ goes.)

\subsection{Crossing matrix of a product of braids}\label{subsec:braidprod}
With this notation we can explain the relation between the crossing matrices \Cross{a}, \Cross{b} of two braids and the crossing 
matrix \Cross{ab} of their product in the braid group, which we think of as represented by the geometric braid $a$ followed by the 
geometric braid $b$. The main observation is that the numbering of the strands of $b$ is changed when we premultiply by $a$:
the $i^{th}$strand of $b$ becomes an extension of the strand of $a$ which landed at the $i^{th}$ position--that is, in $ab$ it
continues the $\of{\invbar{\pi}}{i}^{th}$ strand.  From this it follows that a crossing of the $i^{th}$ strand of $b$ over its $j^{th}$
strand appears in $ab$ as a crossing of the $\of{\invbar{\pi}}{i}^{th}$ strand over the $\of{\invbar{\pi}}{j}^{th}$ strand.  With this 
renumbering of strands in $b$, the crossings add, so we have
\begin{prop} For any two braids $a,b\in\Br{N}$,
	\begin{equation*}
		\Cross{ab}=\Cross{a}+\Cross{b}^{\pis{a}}.
	\end{equation*}
\end{prop}
We will denote this \emph{crossing product} operation on (crossing) matrices by a circled asterisk:
\begin{equation*}
	\bprod{A}{B}\eqdef A+B^{\pis{A}}.
\end{equation*}

\subsection{Order reversal sets}\label{subsec:OR}

In \cite[\S 9.1]{Thurston}, Thurston defined the \deffont{order reversal set} of a permutation $\pi\in\SymGp{N}$:
\begin{equation*}
	\OR{\pi}\eqdef\setbld{(i,j)\in\single{1,2,\dots,N}\times\single{1,2,\dots,N} }{i<j\text{ but }\piof{i}>\piof{j}}.
\end{equation*}
He characterized the sets $S\subset\single{1,2,\dots,N}\times\single{1,2,\dots,N} $ which are order reversal sets for some permutation
$\pi\in\SymGp{N}$ via two properties which are easily seen to be necessary, and with a little more work are sufficient:

\begin{prop}[Thurston]\label{prop:ThurstonOR}
A  subset $S\subset\single{1,2,\dots,N}\times\single{1,2,\dots,N} $ equals \OR{\pi} for some $\pi\in\SymGp{N}$ if and only if
the following properties both hold:
\begin{enumerate}
	\item If $(i,j)\in S$, then for every $k$ between $i$ and $j$, either $(i,k)\in S$ or $(k,j)\in S$ (or both).
	\item Given $i<k<j$, if $(i,k)\in S$ and $(k,j)\in S$, then $(i,j)\in S$.
\end{enumerate}
When these properties hold, the permutation $\pi$ is uniquely determined by $S$.	
\end{prop}

\begin{proof}
	The \textbf{necessity} of these two conditions for an order reversal set is more easily seen when they are replaced with
	their contrapositives:
		\begin{description}
			\item[contra(1)] If, for some $k$ between $i$ and $j$, \emph{neither} $(i,k)$ \emph{nor} $(j,k)$ belongs to $S$,
			then neither does $(i,j)$,
			
			\item[contra(2)] If $(i,j)\not\in S$ given $i<k<j$, then \emph{at most one} of $(i,k)$ and $(k,j)$ belongs to $S$.
			
		\end{description}
		
		For the first, given $i<k<j$, if $\piof{i}<\piof{k}$ and $\piof{k}<\piof{j}$, then of course $\piof{i}<\piof{j}$;
		for the second, $\piof{i}<\piof{j}$ means we can't have both $\piof{k}<\piof{i}$ and $\piof{k}>\piof{j}$.

	To establish \textbf{sufficiency}, we consider the mapping \map{\sigma}{\nlist{N}}{\nlist{N}} defined by the
	analogue of Equation \ref{eqn:perm},
	\begin{equation}\label{eqn:permmap}
		\sigof{i}=i+\sharp\setbld{k}{(i,k)\in S}-\sharp\setbld{k}{(k.i)\in S}.
	\end{equation}
	\begin{description}
		\item[1. $\boldsymbol{(i,j)\in S \Rightarrow\sigof{i}>\sigof{j}}$] 
		
		Given $(i,j)\in S$, the first hypothesis insures that 
		\begin{equation*}
			\sharp\setbld{k<j}{(i,k)\in S}+\sharp\setbld{k>i}{(k,j)\in S}>j-i
		\end{equation*}
		since every $k$ between $i$ and $j$ appears in at least one of these two sets, and $j$ \resp{$i$}
		appears in the first \resp{second} set.  This inequality can be rewritten as
		\begin{align*}
			i+\sharp\setbld{k<j}{(i,k)\in S}&>j-\sharp\setbld{k>i}{(k,j)\in S}.\\
			\intertext{Now, the second hypothesis tells us that for $k>j$, $(j,k)\in S$ implies %
			(since $(i,j)\in S$) that also $(i,k)\in S$, hence}
			\sharp\setbld{k>j}{(i,k)\in S}&\geq\sharp\setbld{k>j}{(j.k)\in S}\\
			\intertext{and similarly, for $k<i$ if $(k,i)\in S$ then also $(k,j)\in S$, hence}
			-\sharp\setbld{k<i}{(k,i)\in S}&\geq-\sharp\setbld{k<i}{(k,j)\in S}.
		\end{align*}
		Adding these inequalities, we obtain
		\begin{multline*}
			i+\sharp\setbld{k<j}{(i,k)\in S}+\sharp\setbld{k>j}{(i,k)\in S}-\sharp\setbld{k<i}{(k,i)\in S}\\
			>j-\sharp\setbld{k>i}{(k,j)\in S}+\sharp\setbld{k>j}{(j.k)\in S}-\sharp\setbld{k<i}{(k,j)\in S}
		\end{multline*}
		or
		\begin{multline*}
			\sigof{i}
			\eqdef 
			i+\sharp\setbld{k}{(i,k)\in S}-\sharp\setbld{k}{(k,i)\in S}\\
			>j-\sharp\setbld{k}{(k,j)\in S}-\sharp\setbld{k}{(k,j)\in S}
			\eqdef
			\sigof{j}.
		\end{multline*}
		
		\item[2. $\boldsymbol{i<j\ \&\ (i,j)\not\in S \Rightarrow\sigof{i}<\sigof{j}}$]
		
		Replacing the first \resp{second} hypothesis with its contrapositive, which is to say
		arguments parallel to the above yield the opposite inequalities and hence the desired conclusion, that for $i<j$,
		if $(i,j)\not \in S$ then $\sigof{i}<\sigof{j}$.
		
	\end{description}
	The definition of $\sigof{i}$ immediately guarantees that $1\leq\sigof{i}\leq N$, and the two inequalities above show that
	$\sigma$ is injective, hence a permutation.  
	Finally, it is easy to see that different permutations have different order reversal sets,
	giving the uniqueness statement in Proposition \ref{prop:ThurstonOR}. 
\end{proof}
Thurston then formulated the  notion of a \emph{permutation braid}:
\begin{definition}\label{dfn:permbraid}
	A \deffont{permutation braid} is a positive braid in which no pair of strands crosses more than once.
\end{definition}
Given a permutation $\pi\in\SymGp{N}$--which is to say, given the ending position of each strand, we can construct a geometric
braid \pbraid{\pi} by joining $(i,1)\in\mathbb{R}^{2}$ to $(\piof{i},0)$ by a straight line segment and making all crossings
positive (in case this yields more than two such line segments crossing at the same point, we can perturb a little to get only pairwise 
crossings). It is clear that no pair of strands crosses more than once.  This shows that  every $\pi\in\SymGp{N}$ is the permutation
associated to some permutation braid;  moreover any permutation braid can be ``straightened out'' so as to be a \pbraid{\pi}.
Thus
\begin{remark}\label{rmk:permbraid} 
	The mapping \map{p}{\SymGp{N}}{\Br{N}} taking a permutation $\pi\in\SymGp{N}$ to the braid \pbraid{\pi} 
	defined above is a bijection onto
	the set of permutation braids.
\end{remark}
We caution the reader that this map is \emph{not a homomorphism}: for example a braid with a single crossing is a permutation 
braid, but its square is not.

\subsection{Characterization of crossing matrices}\label{subsec:crosschar}

The crossing matrix $\mathbf{\Rs{\pi}}$ of the permutation braid \pbraid{\pi} is clearly the same as the matrix
$R$ with $i,j$ entry $1$ if $(i,j)\in \OR{\pi}$ and  $0$ otherwise. 
Since all pairs $(i,j)\in\OR{\pi}$ have $i<j$, $R$ is \textbf{strictly upper triangular} ($\Rs{ij}=0$ for $i\geq j$).
The conditions in  Proposition~\ref{prop:ThurstonOR} characterizing the orientation-reversing set of a permutation 
can be reinterpreted as conditions on the matrix $R$, which we formulate in
\begin{definition}\label{dfn:T}
	We say a square matrix $A$ is $\boldsymbol{T0}$ \resp{$\boldsymbol{T1}$} if
%	\footnote{We have formulated these rules for future application to integer matrices with entries not limited to $0$ and $1$..}
	\begin{description}
		\item[T0] For any triple of indices $i<k<j$, if $\As{ik}=0$ and $\As{kj}=0$, then $\As{ij}=0$.
		
		\item[T1] For any triple of indices $i<k<j$, if $\As{ik}\neq0$ and $\As{kj}\neq0$, then $\As{ij}\neq0$.
	\end{description}
\end{definition}
Note that both conditions refer only to entries above the diagonal of $A$, and can be viewed as limitations on the distribution of zero entries (above the diagonal) in $A$.

\begin{definition}\label{dfn:Rmat}
	An $\boldsymbol{R}$\textbf{-matrix} is a strictly upper triangular matrix, all of whose nonzero entries are $1$, and which is both $T0$ and $T1$.
\end{definition}
For any strictly upper triangular \sqsize{N} matrix $R$, the positions of its nonzero entries form a set $S$ of pairs $(i,j)$ 
with $1\leq i<j\leq N$, and Proposition~\ref{prop:ThurstonOR} tells us when $S$ is an OR set.  
The sufficiency argument in the proof of Proposition~\ref{prop:ThurstonOR} amounts to saying that, for an $R$-matrix,
Equation~\ref{eqn:permmap}  defines a permutation $\sigma\in\SymGp{N}$.
It is easy to see that, for an $R$-matrix, the term in Equation~\ref{eqn:permmap} which is added to $i$ 
equals the $i^{th}$ row sum
\begin{equation*}
	\sharp\setbld{k}{(i,k)\in S}=\sharp\setbld{k>i}{(i,k)\in S}=\sum_{k=1}^{N}\Rs{ik}
\end{equation*}
and the subtracted term equals the $i^{th}$ column sum
\begin{equation*}
	\sharp\setbld{k}{(k.i)\in S}=\sharp\setbld{k<i}{(k.i)\in S}=\sum_{k=1}^{N}\Rs{ki}
\end{equation*}
so Equation~\ref{eqn:permmap} really is the same as Equation~\ref{eqn:perm} applied to the $R$-matrix $R$.
therefore does not change the permutation defined by Equation~\ref{eqn:perm}.  
Thus, a corollary of Proposition~\ref{prop:ThurstonOR} is
\begin{remark}\label{rmk:perm}
	For any \sqsize{N} matrix $A=S+R$, where $S$ is symmetric and $R$ is an $R$-matrix, the formula
	\begin{equation*}
		\pisof{A}{i}= i+\sum_{j=1}^{N}A_{ij}-\sum_{k=1}^{N}A_{ki}
	\end{equation*}
	defines a permutation $\pis{A}\in\SymGp{N}$.
\end{remark}

To characterize the matrices which arise as crossing matrices (of some, not necessarily positive, braid)
we note two necessary conditions:
\begin{prop}\label{prop:basicchar}
	For any crossing matrix $A=((\as{ij}))=\Cross{b}$, $b\in\Br{N}$,
	\begin{enumerate}
		\item the diagonal entries are all zero: 
		\begin{equation*}
			\as{ii}=0 \text{ for }i=1,\dots,N;
		\end{equation*}
		\item each entry above the diagonal is either equal to its symmetric twin, or exceeds it by one: 
		\begin{equation*}
		 	\as{ij}-\as{ji}\in\single{0,1} \text{ for }1\leq i<j\leq N.
		\end{equation*} 
		
	\end{enumerate}
\end{prop}
\begin{proof}
\begin{enumerate}
	\item The first equation is the observation that no strand crosses itself.
	
	\item Suppose $1\leq i<j\leq N$.  Since strand $i$ starts to the left of strand $j$, if any crossings of these two strands occur,
	the \emph{first} one mst move $i$ to the right of $j$
	either via a \emph{positive} crossing of $i$ \emph{over} $j$, or via a
	\emph{negative} crossing of $j$ over $i$; 
	the effect of this  is to contribute an \emph{increase} by one to the difference $\as{ij}-\as{ji}$.
	A \emph{second} crossing must move $i$ back to the left of $j$, 
	either via a \emph{negative} crossing of $i$ over $j$ 
	or via a \emph{positive} crossing of $j$ over $i$;
	this \emph{decreases} the difference $\as{ij}-\as{ji}$ by one.
	
	As long as crossings of $i$ with $j$ continue, these two situations will alternate strictly, 
	so the difference $\as{ij}-\as{ji}$ will oscillate between $0$ and $1$.
\end{enumerate}
\end{proof}

A braid is $b$ called \deffont{pure} if its permutation $\pis{b}$ (as defined in \S\ref{subsec:braidprod})
 is the identity permutation.  It is easy to see that this condition can be expressed by saying that \emph{the $i^{th}$
 row sum equals the $i^{th}$ column sum} for $i=1,\dots,N$.  If for each pair of indices $i,j$ with 
 $1\leq i<j\leq N$ we set $\xs{ij}=\as{ij}-\as{ji}$, this condition becomes the system of $N-1$ equations in $\frac{(N-1)(N-2)}{2}$ 
 unknowns of the form $\sum_{j=i+1}^{N}\xs{ij}=0$, $i=1,\dots,N-1$.  This system has many integer solutions, for example 
 the matrix
 \begin{equation*}
 	\left[\begin{array}{ccc}
		0 & 2 & 0 \\
		1 & 0 & 2 \\
		1 & 1 & 0
	\end{array}\right]
 \end{equation*}
 has each row sum equal to the corresponding column sum,
 but if we also throw
 in the earlier requirement that $\as{ij}-\as{ji}\eqdef\xs{ij}\in\single{0,1}$ we see that the only solution is $\xs{ij}=0$ for all $i<j$.  Thus
 \begin{lemma}\label{lem:puresym}
 	The crossing matrix of every pure braid is symmetric.
 \end{lemma}
  We note that the pure braids form a subgroup \Pur{N} of the braid group \Br{N}, and that the restriction of the crossing matrix
 mapping to pure braids is a homomorphism to the additive group \SymMat{N}
 of symmetric \sqsize{N} integer matrices with zero diagonal. This map is also surjective;  
 to see this, note that each of the symmetric matrices $\Ssub{ij}$ whose only nonzero entries are
 a ``$1$'' in the $i,j$ and $j,i$ positions is the crossing matrix of the braid \ssub{ij} in which strand $i$ crosses \emph{over} all intermediate 
 strands, ``hooks'' strand $j$, and then crosses back \emph{over} all the intermediate strands%
 \footnote{
 (Note that the latter set of crossings is \emph{negative}.)
 }
  to return to its initial position.  

Since the \sqsize{N} matrices $\Cross{\ssub{ij}}=\Ssub{ij}$ for $1\leq i<j\leq N$ generate the additive group \SymMat{N}, this shows 
\begin{prop}\label{prop:PontoS}
	The crossing matrix map takes the subgroup \Pur{N} of pure $N$-strand  braids onto the additive group \SymMat{N}
	of \sqsize{N} symmetric integer matrices with zero diagonal:
	\begin{equation*}
		\Cross{\Pur{N}}=\SymMat{N}.
	\end{equation*}
\end{prop}

To extend our characterization of crossing matrices to \emph{all} braids, we note that if $b\in\Br{N}$ is a braid with permutation \pis{b},
then the braid $s(b)\eqdef b(\pbraid{\pis{b}})^{-1}$ is a pure braid, 
and so we have a unique factoring $b=s(b)\pbraid{\pis{b}}$ as a pure braid followed by
a permutation braid.  It follows that we can write the crossing matrix of $b$ as
\begin{equation*}
	\Cross{b}=\Cross{s(b)\pbraid{\pis{b}}}=\bprod{S}{R}=S+R,
\end{equation*}  
 where $S=\Cross{s(b)}\in\SymMat{N}$ (so $\pis{S}=id$) and $R=\Rs{\pis{b}}$ is an $R$-matrix.
Since $R$-matrices are upper triangular, the expression $A=S+R$ is uniquely determined (if it exists) for any matrix; we will refer to
it as the \deffont{$\boldsymbol{SR}$ decomposition} of $A$..
We can therefore complete our characterization of crossing matrices for (general) braids:
\begin{theorem}\label{thm:CofB}
	An \sqsize{N} matrix $A$ is the crossing matrix of some braid if and only if it has an $SR$ decomposition.
\end{theorem}
We denote the set of all \sqsize{N} integer matrices with zero diagonal which have an $SR$ decomposition by $\boldsymbol{\SR{N}}$.

It can sometimes be difficult to immediately visualize the $SR$ decomposition of an integer matrix, so we have 
adopted a tableau notation to make this clearer: given an \sqsize{N} matrix with an $SR$ decomposition $A=S+R$, 
we exhibit each entry of $A$ $i,j$ strictly above the diagonal ($1\leq i<j\leq N$) as the sum of the corresponding entries \ssub{ij} 
of $S$ and \rs{ij} of $R$, in the form ``$\ssub{ij}S+\rs{ij}R$''; note that 
\begin{align*}
	\ssub{ij}&=\as{ji}\\
	\rs{ij}&=\as{ij}-\as{ji}.
\end{align*}
For example, if
\begin{equation*}
	A=\left[\begin{array}{cccc}
			0 & 1 & 3 & 1 \\
			0 & 0 & 0 & 1 \\
			2 & 0 & 0 & 0 \\
			0 & 1 & -1 & 0
		\end{array}\right]
		=\left[\begin{array}{cccc}
			0 & 0 & 2 & 0 \\
			0 & 0 & 0 & 1 \\
			2 & 0 & 0 & -1 \\
			0 & 1 & -1 & 0
		\end{array}\right]
		+
		\left[\begin{array}{cccc}
			0 & 1 & 1 & 1 \\
			0 & 0 & 0 & 0 \\
			0 & 0 & 0 & 1 \\
			0 & 0 & 0 & 0
		\end{array}\right]
\end{equation*}
then the \deffont{$\boldsymbol{SR}$ tableau} encoding this data is
\begin{equation*}
	\SRT{A}=\fourTableau{R}{2S+R}{R}{0}{S}{-S+R}
%\begin{array}{cccc}
%	\fbox{1}&R&2S+R&R\\
%	&\fbox{2}&0&S\\
%	&&\fbox{3}&-S+R\\
%	&&&\fbox{4}.
%\end{array}
\end{equation*}

To complete this picture, we note that the crossing matrix map is not injective.  
First, we observe that the restriction of the crossing matrix map to the subgroup \Pur{N} is the abelianization of that group;
it follows that its kernel is the commutator subgroup of \Pur{N}.  In general, two braids $a$ and $b$ with the same crossing
matrix will have the same permutation (according to the formula in Equation~\ref{eqn:perm})
so $c=\invbar{a}b$ is a pure braid with zero crossing matrix;  it follows that
$b=ac$, where $c$ belongs to the
commutator subgroup of the pure braid group--that is, $c$ can be written as a product of finitely many braids of the form
$[ p,q ]\eqdef\ps{i}\qs{i}\invbar{\ps{i}}\invbar{\qs{i}}$.  
%Furthermore, the ``full twist'' braid \ftwist{} (obtained by ``rotating'' the strands a
%full turn as they move down) is a positive pure braid which generates the center of the braid group. In particular, replacing either 
%$p$ or $q$ with $\ftwist p$ \resp{$\ftwist q$} doesn't change the commutator.  It is known \cite{Garside} that any braid $b$
%can be written in the form  
We formalize this observation:
\begin{prop}\label{prop:kerC}
	Two braids $a,b\in\Br{N}$ satisfy $\Cross{a}=\Cross{b}$ 
	if and only if $a=bc$, where $c$ belongs to the commutator subgroup of \Pur{N};  that is,
	$c=[\ps{1},\qs{1}]\cdots[\ps{k},\qs{k}]$ where $\ps{i},\qs{i}\in\Pur{N}$ for $i=1,\dots,k$.
\end{prop}

\section{Crossing Matrices of Positive Braids}\label{sec:CMP}

We turn now to the focus of this paper, the crossing matrices of \emph{positive} braids; we will adapt all of our notation to this case
by using a superscript ``plus'' to denote positivity:  \Brp{N} \resp{\Purp{N}} will denote the \emph{positive} \resp{pure positive} braids.  Again, our interest in this special case is
prompted by the fact that there is no cancellation of crossings in the crossing matrix: for $b\in\Brp{N}$, every crossing is accounted
for in \Cross{b}.

\subsection{Two examples}\label{sec:examples}
It would be natural to expect, in view of Theorem~\ref{thm:CofB}, that \Cross{\Brp{N}} consists of all matrices in
\SR{N} with all entries non-negative.  However, this is false;  the 
following are examples of elements of \SR{N} which, while they have non-negative entries and 
are crossing matrices of \emph{some} braids%
%(see Figures~\ref{fig:G} and \ref{fig:K})
, are not crossing matrices of any \emph{positive} braids.

It will prove easier to understand these and other examples using their $SR$ tableaux:
\begin{equation*}
	\SRT{G}=\fourTableau{R}{S}{0}{0}{S}{R}
%		\begin{array}{cccc}
%			\fbox{1}&R&S&0\\
%			&\fbox{2}&0&S\\
%			&&\fbox{3}&R\\
%			&&&\fbox{4}
%		\end{array}
\end{equation*}
and
\begin{equation*}
	\SRT{K}=
	\begin{array}{ccccc}
		\hline
		\multicolumn{1}{|c}{\labbox{1}} & 
		\multicolumn{1}{|c}{0} & 
		\multicolumn{1}{|c}{S} & 
		\multicolumn{1}{|c}{0} &
		\multicolumn{1}{|c|}{0}
		\\\hline 
	& 
		\multicolumn{1}{|c}{\labbox{2}} & 
		\multicolumn{1}{|c}{R} &
		\multicolumn{1}{|c}{R} &
		\multicolumn{1}{|c|}{0} 
		\\\cline{2-5}
	& &
		\multicolumn{1}{|c}{\labbox{3}} & 
		\multicolumn{1}{|c}{R} &
		\multicolumn{1}{|c|}{S} 
		\\\cline{3-5}
	& &&
		\multicolumn{1}{|c}{\labbox{4}} & 
		\multicolumn{1}{|c|}{0}  
		\\\cline{4-5}
	& & & &
		\multicolumn{1}{|c|}{\labbox{5}} 
		\\\cline{5-5}
	\end{array}%
%		\begin{array}{ccccc}
%			\fbox{1}&0&S&0&0\\
%			&\fbox{2}&R&R&0\\
%			&&\fbox{3}&R&S\\
%			&&&\fbox{4}&0\\
%			&&&&\fbox{5}.
%		\end{array}
\end{equation*}

We will explain why these two matrices are not crossing matrices of any positive braids after developing some properties of such
crossing matrices.

\subsection{The positive realization problem}\label{subsec:bprops}

The positive braids form a semigroup which includes all permutation braids, so any product of permutation braids, while it may no 
longer be a \emph{permutation} braid, will be a \emph{positive} braid.  Conversely, since the standard generators of the braid group 
are the permutation braids for transpositions of adjacent strands and a positive braid is given by a positive word in these generators,
any realization of a matrix $A$ as the crossing
matrix of some positive braid can always be expressed as a product of permutation braids, 
and this corresponds to a factoring  of $A$ into a product (with respect to the crossing product operation \bmult) of $R$-matrices.

\begin{remark}\label{rmk:factoring}
	An \sqsize{N} matrix $A$ is the crossing matrix of some positive braid if and only if it can be factored as a product 
	\begin{equation*}
		A=\Rs{\pis{1}}\bmult\Rs{\pis{2}}\bmult\cdots\bmult\Rs{\pis{k}}
	\end{equation*}
	of $R$-matrices.
\end{remark}

Our problem, then, is how to tell, given a matrix $A$, whether such a factoring is possible.
Since any permutation braid is a product of single-crossing braids (\emph{i.e.,} any permutation is a product of transpositions)
we can focus on factorization into crossing matrices of transpositions.%
\footnote{
	By abuse of notation, we will refer to  ``factoring into transpositions''.
} 

We note that in addition to the fact that all entries in the crossing matrix of a positive braid are non-negative integers,
there is an immediate further restriction on such crossing matrices: 
\begin{remark}\label{rmk:T0}
	If a braid is positive, $b\in\Brp{N}$, then its crossing matrix \Cross{b} must
	\begin{enumerate}
		\item have non-negative integer entries and
		\item satisfy the $T0$ property (Definition~\ref{dfn:T})
	\end{enumerate}
\end{remark}
To see the second requirement, note that the argument for $T0$  in the context of $R$-matrices 
(the necessity of the first condition in Proposition~\ref{prop:ThurstonOR}) 
is based only on the assumption that permutation braids are positive.
By contrast, the argument for $T1$ in Proposition~\ref{prop:ThurstonOR} 
is based on the assumption that in a permutation braid no pair of strands crosses more than once; 
this is not necessarily the case for general positive braids.  

In view of Remark~\ref{rmk:T0}, we define $\boldsymbol{\SRp{N}}$ to be 
the set of (non-negative integer, zero diagonal) $T0$ matrices which can be written
as the sum of a (non-negative integer) symmetric matrix and an $R$-matrix:
\begin{equation*}
	\SRp{N}\eqdef
	\setbld{A\in\mathbb{Z}^{+}}%
		{A \text{ is }T0 \text{ and }A=S+R, \text{ where }S\in\SymMat{N}\text{ and }R\text{ is an }R-\text{matrix}} 
\end{equation*}
We caution that the $T0$ property for $A$ does not necessarily require $S$,  the symmetric part of the decomposition, to be $T0$
on its own (unless, of course, $R=0$, so the braid is pure).

Note that the two examples in Section~\ref{sec:examples} fit this description: $G\in\SRp{4}$ and $K\in\SRp{5}$.

\subsection{Left division by permutations}\label{subsec:multbyperm}
To study the factorization problem posed by  Remark~\ref{rmk:T0}, we formulate a partial inverse to the crossing product operation \bmult. 
If we solve the equation
\begin{equation*}
	A=\bprod{\Rs{\pi}}{B}\eqdef\Rs{\pi}+B^{\pi}
\end{equation*}
for $B$ in terms of $A$, we obtain a ``division on the left'', defined by
\begin{align*}
	B&=\ldiv{\Rs{\pi}}{A}\eqdef(A-\Rs{\pi})^{\invbar{\pi}}\\
	\intertext{which we will usually refer to as \deffont{left division by $\boldsymbol{\pi}$}:}
	B&=\ldiv{\pi}{A}.
\end{align*}
If $A$ and $B$ are to be crossing matrices of positive braids, we must check conditions which insure, for $A\in\SRp{N}$,
that $B\in\SRp{N}$.  

The first requirement is the obvious one, that all entries of the matrix $A-\Rs{\pi}$ are non-negative;
this is the same as requiring that $\Rs{\pi}\leq A$ entrywise.  When this is the case, we will say that the permutation $\pi$
is \deffont{subordinate} to $A$. 
	
The second requirement is that $B$ be $T0$.  To understand how this can fail, 
we take a detour and develop a way to ``localize'' some of the effects
of multiplication (\bprod{\Rs{\pi}}{B}) and division (\ldiv{\pi}{A}) on a matrix $B\in\SRp{N}$ \resp{$A\in\SRp{N}$}.

\subsubsection{Configurations}\label{subsubsec:config}
Given a braid $b$ on $N$ strands, we can pick out a proper subset \subind{} of $m<n$ strands 
and study the relations just between these selected strands by considering the ``subbraid'' $\bs{\subind}\in\Br{m}$
which results from erasing all the other strands of $b$.  The (crossing) matrix analogue of this is, given $A\in\SRp{N}$,
to look at the \sqsize{m} submatrix \AI{}  consisting of elements whose indices are both in \subind.  
Note that if $A\in\SRp{N}$ then $\AI\in\SRp{m}$.  

	We can abstract this:
	\begin{definition}[Configuration]
		Given an \sqsize{N} matrix $A$ and an index set 
		\begin{equation*}
			\subind=\single{\ks{1},\ks{2},\dots,\ks{m}}\subset\nlist{N}\quad(m\leq N), 
		\end{equation*}
		the \textbf{configuration} of $A$ on \subind{}
		is the \sqsize{m} matrix $\boldsymbol{\AI}$ consisting of the rows and columns of $A$ whose indices belong to \subind.
	\end{definition}

	When $A$ is acted on by a permutation $\pi$ (for example as part of a  product or left division), the entries of
	\AI{} change position, in two ways: 
	the \emph{set} of new positions is the image $\piof{\subind}$ of \subind{} (as a mapping
	of \nlist{N} to itself), 
	and their \emph{relative order} may be scrambled;
	the scrambling action is given by the permutation \piI{} whose order-reversal set is the intersection of \OR{\pi} with
	$\subind\times\subind$: 
	\begin{equation*}
		\OR{\piI}=\OR{\pi}\cap(\subind\times\subind).
	\end{equation*}
	We refer to the permutation \piI{} (by abuse of terminology) as the \textbf{restriction of $\boldsymbol{\pi}$} to \subind.

	It is easy to confirm that the total effect of $\pi$ on configurations is given by %the formula
%	\begin{equation*}
%		(A^{\pi})_{\piof{\subind}}=(\AI)^{\piI}.
%	\end{equation*}
%	In particular, if \OR{\pi} includes no \emph{pair} of indices from \subind{}, we can say that the configuration moves but 
%	does not change: $(A^{\pi})_{\piof{\subind}}=\AI$.
%	
%
\begin{remark}\label{rmk:submatperm}
For any \sqsize{N} matrix $A$, any subset $\subind\subset\single{1,\dots,N}$ 
and any permutation $\pi$ of \single{1,\dots,N},
	\begin{equation*}
		(A^{\pi})_{\piof{\subind}}=(\AI)^{\piI}.
	\end{equation*}
\end{remark}
The point here is that, although the set of indices \subind{} is not in general invariant under $\pi$, 
it is still true that if $\ks{i}<\ks{j}$ are two elements of \subind{},
then $\piof{\ks{i}}>\piof{\ks{j}}$ \emph{only} if $\of{\piI}{i}>\of{\piI}{j}$. 
%	In particular, if \OR{\pi} includes no \emph{pair} of indices from \subind{}, we can say that the configuration moves but 
%	does not change: $(A^{\pi})_{\piof{\subind}}=\AI$.

Remark~\ref{rmk:submatperm} lets us trace a configuration of crossings corresponding to a sub-braid
through operations like ``division by a permutation'' without regard to how crossings outside that
configuration are affected.   

\begin{remark}\label{rmk:permsubmat}
	Given a  permutation $\pi$ and an index set $\subind=(\ks{1},\dots,\ks{m})$,
	\begin{equation*}
		\Rs{\piI}=(\Rs{\pi})_{\subind}.
	\end{equation*}
	In particular, given a matrix $A$, a permutation $\pi$ and the index set \subind,  the effect 
	on the configuration \AI{} corresponding to \subind{} 
	of multiplying $A$ by $R=\Rs{\pi}$
	is given by
	\begin{equation}\label{eqn:configprod}
		\lbrack(\bprod{R}{A})_{\subind^{\pi}}\rbrack_{\pis{i},\pis{j}}=\As{i,j}+\Rs{i,j}\quad\text{ for }
		i,j\in\subind
	\end{equation}
	while the effect of left-dividing $A$ by $\pi$ (provided $\pi\leq A$) is expressed %in terms of $R^{-}=\Rs{\bpiI}$
	by
	\begin{equation}\label{eqn:configldiv}
		\lbrack(\ldiv{\pi}{A})_{\subind^{\pi}}\rbrack_{\bpis{i},\bpis{j}}=\As{i,j}-\Rs{i,j}\quad\text{ for }
		i,j\in\subind.
	\end{equation}
\end{remark}

A consequence of Remark~\ref{rmk:permsubmat} is that for any configuration \AI{} of $A$, if the restriction
\piI{} to \subind{} of a permutation $\pi$ is the identity (that is, its order reversal set is disjoint from \subind),
then the configurations $(\bprod{\Rs{\pi}}{A})_{\subind^{\pi}}$ and   $(\ldiv{\pi}{A})_{\subind^{\pi}}$ both equal
the configuration \AI{}--in other words, the relative positions of these entries do not change, even though this
matrix may be embedded in the corresponding larger matrix \bprod{\Rs{\pi}}{A} \resp{\ldiv{\pi}{A}} in different ways.

\subsubsection{Multiplication and division by transpositions}\label{subsubsec:divtran}
Any permutation can be expressed as a product of transpositions of adjacent positions.  We denote by
$\boldsymbol{\taus{i}}\in\SymGp{N}$  ($1\leq i<n$) the permutation which interchanges positions $i$ and $i+1$
and leaves every other position alone.  Its crossing matrix $\boldsymbol{\ts{i}}=\Rs{\taus{i}}$ is the matrix with $1$
in position $i,i+1$ and zero everywhere else. Note that \taus{i} is its own inverse.%
\footnote{Caution: this is not true of the corresponding permutation braid.} 

\textit{Multiplication by \taus{i}}: Given $A\in\SRp{N}$, \bprod{\ts{i}}{A} has rows \resp{columns} $i$ and $i+1$ interchanged, in
particular the \deffont{subdiagonal} entry of \bprod{\ts{i}}{A} in position $i+1,i$ is the same as the \deffont{superdiagonal}
%\footnote{We refer to positions in $A$ with column number one more \resp{one less} than the row number 
%as the \deffont{first superdiagonal} \resp{\deffont{first subdiagonal}} of $A$.}
 entry at $i,i+1$ in A, while 
the $i,i+1$ entry of \bprod{\ts{i}}{A} is one more than the $i+1,i$ entry of $A$.  In terms of the $SR$ tableau, 
again aside from the $i,i+1$ entry, multiplication by \ts{i} interchanges rows $i$ and $i+1$ with each other, and columns $i$ and $i+1$ with each other.
In position $i,i+1$, a zero \resp{R} in $A$ becomes an R \resp{S} in \bprod{\taus{i}}{A}.

\textit{Division} by \taus{i}:
The transposition \taus{i} is subordinate to the matrix $A\in\SRp{N}$ if and only if the entry \as{i,i+1} on the first superdiagonal of $A$
is nonzero, and in this case  left division by \taus{i} also interchanges rows \resp{columns}  $i$ and
$i+1$ except that the (subdiagonal) $i+1,i$ entry of \ldiv{\taus{i}}{A} is one less than the  (superdiagonal) $i,i+1$ entry of $A$.
In terms of the $SR$ tableau, division by \taus{i} also interchanges the $i^{th}$ an $(i+1)^{st}$ rows \resp{columns} of A  
and, in the $i,i+1$ position of the tableau, changes an S to an R or an R to a zero.
                                                              
\subsubsection{Mirror symmetry}\label{subsubsec:symm}
It will  simplify some arguments to note that the symmetry on \sqsize{m} matrices 
defined by
\begin{equation*}
	(\mir{A})_{i,j}\eqdef A_{(m+1-i),(m+1-j)}
\end{equation*} 
preserves realizability.  
This consists of reversing 
the order of the rows (and columns): it is the analogue of transpose, but ``flips''
the matrix about its antidiagonal instead of its diagonal.  
We will refer to \mir{A} as the ``mirror'' of $A$.  
It is useful to have this operation not just for the full \sqsize{N} matrices, but for their 
sub matrices as well (which is why it is defined above for \sqsize{m} instead of just \sqsize{N} 
matrices).

 If $M=C(b)$, then $\mir{M}=C(b^{\star})$, where
$b^{\star}$ is the geometric braid obtained by looking at $b$ from ``behind'';  equivalently,
 $b^{\star}$ is obtained from
$b$ by numbering the strands right-to-left instead of left-to-right.
(Note that a positive
crossing remains positive if viewed from ``behind''.)
Thus, a matrix $A\in\SRp{N}$ is realizable as the crossing matrix of a (positive) braid if and only if \mir{A} is.

\subsubsection{Creating $T0$ violations}\label{subsubsec:createT0viol}

We want to understand how left division of $A\in\SRp{N}$ by a transposition $\pi=\taus{i}$ subordinate to $A$ results in 
$\ldiv{\pi}{A}\not\in\SRp{N}$, a situation which we have seen can only happen if $B=\ldiv{\pi}{A}$ fails to be $T0$.  
A \textbf{violation of $\boldsymbol{T0}$ }
means a \sqsize{3} configuration in $B$ of the form
\begin{equation*}
	\Bs{(\subind)^{\piI}}
	=\left\lbrack\begin{array}{ccc}
		0 & 0 & b \\
		a & 0 & 0 \\
		\bp & \ap & 0
	\end{array}\right\rbrack
\end{equation*}
with $b\neq0$. 
If $\piof{\subind}=\single{\ps{1},\ps{2},\ps{3}}$ with $\ps{1}<\ps{2}<\ps{3}$ and $A\in\SRp{N}$ (so that the corresponding 
configuration \As{\subind} is different from $\Bs{\piof{\subind}}$), then $\pi$, if it is a transposition, must interchange 
either \ps{1} and \ps{2}, or \ps{2} and \ps{3}.  This requires that the interchanged pair of indices be adjacent.  In short,
there are only two possible ways that  left dividing $A$ by a subordinate transposition can change \As{\subind}: either 
$\piI=\taus{1}$ (and $\ps{2}=\ps{1}+1$) or $\piI=\taus{2}$ (and $\ps{3}=\ps{2}+1$).  Notice that these two situations are
mirrors of each other, so we can concentrate on the case that $\piI=\taus{1}$ (and $\ps{2}=\ps{1}+1$).  Then it follows
from Remark~\ref{rmk:permsubmat} that 
\begin{equation*}
	\As{\subind} = \bprod{\taus{1}}{\Bs{\piof{\subind}}}
	=\left[\begin{array}{ccc}
		0 & 1 & 0 \\
		0 & 0 & 0 \\
		0 & 0 & 0
	\end{array}\right]
	\bmult
	\left\lbrack\begin{array}{ccc}
		0 & 0 & b \\
		a & 0 & 0 \\
		\bp & \ap & 0
	\end{array}\right\rbrack
	=
	\left[\begin{array}{ccc}
		0 & a+1 & 0 \\
		0 & 0 & b \\
		\ap & \bp & 0
		\end{array}\right].
\end{equation*}
Since $A$ (and hence \As{\subind}) has an $SR$ decomposition, we must have $a=\ap=0$ and $b=\bp$ or  $\bp+1$;
thus
\begin{equation*}
	\As{\subind}=
	\left[\begin{array}{ccc}
		0 & 1 & 0 \\
		0 & 0 & b \\
		0 & \bp & 0
	\end{array}\right]
\end{equation*}
with $b=\bp$ or  $b=\bp+1$.  This means the $SR$ tableau of \As{\subind} is one of two possibilities:
\begin{equation*}
	\SRT{\As{\subind}}=\threeConfig{\ps{1}}{\ps{2}}{\ps{3}}{R}{0}{bS}
%		\begin{array}{ccc}
%			\fbox{1}&R&0\\
%			&\fbox{2}&bS\\
%			&&\fbox{3}
%		\end{array}
		\quad\text{or}\quad
		\threeConfig{\ps{1}}{\ps{2}}{\ps{3}}{R}{0}{bS+R}
%		\begin{array}{ccc}
%			\fbox{1}&R&0\\
%			&\fbox{2}&bS+R\\
%			&&\fbox{3}
%		\end{array}.
\end{equation*}
But the second tableau violates the requirement that the R-matrix in the $SR$ decomposition must be $T1$, so only the first
can belong to \SRp{N}.

We therefore adopt the following 
\begin{definition}\label{dfn:blockage}
	A \textbf{blockage} in the matrix $A\in\SRp{N}$ is a \sqsize{3} configuration of the form
	\begin{align*}
		\SRT{\AI}&=\threeConfig{i}{j}{k}{R}{0}{bS}\\
		\intertext{or its mirror image}\\
		\SRT{\AI}&=\threeConfig{i}{j}{k}{bS}{0}{R}.
	\end{align*}
	with $b\neq0$.
	In either of these situations, we say that the ``R'' entry is \textbf{blocked} (by the ``S'' entry) .
\end{definition}

The preceding discussion shows that division of a matrix in \SRp{N} by a subordinate transposition results in a $T0$ violation
precisely if that transposition corresponds to a ``blocked R'' in the matrix:
 
\begin{prop}[Division by a Blocked ``R'']\label{prop:blockage}
	If $B=\ldiv{\tau}{A}\not\in\SRp{N}$ where $A\in\SRp{N}$ and $\tau\in\SymGp{N}$ is a transposition subordinate to $A$, 
	then $A$ has a blockage 
	\AI{} whose ``R'' entry lies on the first superdiagonal of $A$, at the position corresponding to $\tau$: either
	\begin{enumerate}
		\item $\tau=\taus{i}$, $j=i+1$ and
		\begin{equation*}
			\SRT{\AI}=\threeConfig{i}{i+1}{k}{R}{0}{bS},
		\end{equation*}
		or
		\item $\tau=\taus{j}$, $k=j+1$ and
		\begin{equation*}
			\SRT{\AI}=\threeConfig{i}{j}{j+1}{bS}{0}{R}
		\end{equation*}.
	\end{enumerate}
	
	In either case, the resulting $T0$ violation is
	\begin{equation*}
			\SRT{\Bs{\subind}}=\threeConfig{i}{j}{k}{0}{bS}{0}.
	\end{equation*}
\end{prop}
 
Proposition~\ref{prop:blockage} can also be understood geometrically, in terms of possible realizations of the matrix. 
Suppose $A\in\SRp{N}$ has a blockage, not necessarily embedded with the ``$R$'' on the first superdiagonal of $A$.  
Then, in any possible realization of $A$ as the crossing matrix of a positive braid,
the subbraid corresponding to this configuration consists of three strands, with the middle strand ``hooking'' one of the outer 
strands but crossing the other outer strand only once.  In such a (sub)braid, all the ``hooks'' must precede the crossing, because 
once the crossing has occurred, the formerly outside strand separates the formerly middle strand from the other outer strand
(see Figure~\ref{fig:blockage}).  

In the situation where the ``R'' \emph{does} lie on the first superdiagonal, this reasoning shows
that attempting at that stage to introduce the transposition corresponding to that ``R'' cannot lead to a realization of the matrix. 
However, if via other divisions we can move the blockage so that the ``S'' entry is on the first superdiagonal, then we can 
divide by that transposition, thereby destroying the blockage, and continue.

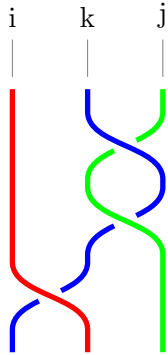
\begin{figure}[h]
\begin{center}
\begin{tikzpicture}
	\braid[
	line width=2pt,
	style strands={1}{red},
	style strands={2}{blue},
	style strands={3}{green},
	]
	(braid) at (2,0) s_{2} s_{2} s_{1} ;
	\node[at=(braid-1-s),pin=north : i] {};
	\node[at=(braid-2-s),pin=north : k] {};
	\node[at=(braid-3-s),pin=north : j] {};

\end{tikzpicture}

\caption{Geometric Interpretation of a Blockage}
\label{fig:blockage}
\end{center}
\end{figure}

%This reasoning can be extended to the case when the upper right entry is zero and both superdiagonal entries have multiples
%of ``$S$'', but (only) one of them has an ``$R$'' as well;  
%in this case, \emph{all} of the crossings yielding the various ``$S$'' entries must occur
%\emph{before} the one crossing associated with the ``$R$'';  furthermore, there can be no intermingling of the various ``hooks'': 
%any given hook, once begun, must be completed before any of the other crossings in this configuration can occur.  Thus, the
%``hooks'' \emph{block} the single crossing--they must be divided out (in any factoring of $A$ into permutation matrices) before the 
%single crossing can be realized.  In different words, a blockage cannot be divided by a transposition corresponding to the ``$R$''
%entry.

The two examples in Subsection~\ref{sec:examples} both have the property that 
\emph{every nonzero entry in the first superdiagonal is a blocked ``$R$''}: the configurations \Gs{124}, \Gs{134}, \Ks{235} and
\Ks{134} are all blockages whose ``$R$'' is embedded in the first superdiagonal of the respective matrix.  Therefore, it is 
impossible to factor either of these matrices into permutation matrices--or equivalently, neither matrix is the crossing matrix 
of any positive braid.  We call a matrix $A\in\SRp{N}$ with this property a \deffont{totally blocked} matrix: 
it is impossible to left divide it by a transposition subordinate to the matrix (see below).

\begin{remark}\label{rmk:totallyblocked}
	If $A\in\SRp{N}$ is totally blocked--that is, every nonzero entry in the first superdiagonal of $A$ is a blocked ``$R$''--
	then $A $ is not the crossing matrix of any positive braid.
\end{remark}

As we pointed out above, the presence of a blockage somewhere in $A\in\SRp{N}$ 
does not \textit{a priori} mean that $A$ is not the crossing
matrix of some positive braid--in fact, there are factorizations (equivalently, sequences of left divisions by transpositions)
which create blockages, but then these blockages move around the matrix until they land with the ``$S$'' term on the first 
superdiagonal.  

When the ``S'' entry lies on the first superdiagonal, division by the transposition subordinate to it yields a
configuration of the form
\begin{equation*}
	\threeConfig{i}{j}{j+1}{0}{R}{R+(b-1)S}
\end{equation*}
(or its mirror image) and a further division by the same transposition yields
\begin{equation*}
	\threeConfig{i}{j}{j+1}{R}{0}{(b-1)S}
\end{equation*}
\resp{its mirror image};  repeating this process $2(b-1)$ more times ends in the configuration
\begin{equation*}
	\threeConfig{i}{j}{j+1}{0}{R}{R}
\end{equation*}
thus eliminating the blockage.
%Furthermore, a blockage with no entries on the first superdiagonal will move without changing as long as we are dividing by
%transpositions that do not include its entries.

However, a matrix $A\in\SRp{N}$ need not be totally blocked to fail to be a crossing matrix of some positive braid.
Consider for example the \sqsize{6}  matrix
\begin{equation*}
	V=
	\left[\begin{array}{cccccc}
		0 & 1 & 1 & 1 & 0 & 0\\
		1 & 0 & 1 & 0 & 1 & 0\\
		0 & 0 & 0 & 0 & 1 & 1\\
		1 & 0 & 0 & 0 & 1 & 1 \\
		0 & 0 & 1 & 0 & 0 & 0\\
		0 & 0 & 1 & 1 & 0 & 0
	\end{array}\right]
\end{equation*}
with $SR$ tableau
\begin{equation*}
	\SRT{V}=
	\begin{array}{cccccc}
		\hline
		\multicolumn{1}{|c}{\labbox{1}} & 
		\multicolumn{1}{|c}{S} & 
		\multicolumn{1}{|c}{R} & 
		\multicolumn{1}{|c}{S} &
		\multicolumn{1}{|c}{0} &
		\multicolumn{1}{|c|}{0}
		\\\hline 
	& 
		\multicolumn{1}{|c}{\labbox{2}} & 
		\multicolumn{1}{|c}{R} & 
		\multicolumn{1}{|c}{0} &
		\multicolumn{1}{|c}{R} &
		\multicolumn{1}{|c|}{0} 
		\\\cline{2-6}
	& &
		\multicolumn{1}{|c}{\labbox{3}} & 
		\multicolumn{1}{|c}{0} & 
		\multicolumn{1}{|c}{S} &
		\multicolumn{1}{|c|}{S} 
		\\\cline{3-6}
	& &&
		\multicolumn{1}{|c}{\labbox{4}} & 
		\multicolumn{1}{|c}{R} & 
		\multicolumn{1}{|c|}{S} 
		\\\cline{4-6}
	& &&&
		\multicolumn{1}{|c}{\labbox{5}} & 
		\multicolumn{1}{|c|}{0}  
		\\\cline{5-6}
	& & & &&
		\multicolumn{1}{|c|}{\labbox{6}} 
		\\\cline{6-6}
	\end{array}\quad.%
\end{equation*}
The two ``$R$'''s on the first superdiagonal are blocked (\Vs{236}, \Vs{145}).  
Thus the only allowed left division is by \taus{1}, and this results in 
\begin{equation*}
	\SRT{V1=\ldiv{\taus{1}}{V}}=
	\begin{array}{cccccc}
		\hline
		\multicolumn{1}{|c}{\labbox{1}} & 
		\multicolumn{1}{|c}{R} & 
		\multicolumn{1}{|c}{R} & 
		\multicolumn{1}{|c}{0} &
		\multicolumn{1}{|c}{R} &
		\multicolumn{1}{|c|}{0}
		\\\hline 
	& 
		\multicolumn{1}{|c}{\labbox{2}} & 
		\multicolumn{1}{|c}{R} & 
		\multicolumn{1}{|c}{S} &
		\multicolumn{1}{|c}{0} &
		\multicolumn{1}{|c|}{0} 
		\\\cline{2-6}
	& &
		\multicolumn{1}{|c}{\labbox{3}} & 
		\multicolumn{1}{|c}{0} & 
		\multicolumn{1}{|c}{S} &
		\multicolumn{1}{|c|}{S} 
		\\\cline{3-6}
	& &&
		\multicolumn{1}{|c}{\labbox{4}} & 
		\multicolumn{1}{|c}{R} & 
		\multicolumn{1}{|c|}{S} 
		\\\cline{4-6}
	& &&&
		\multicolumn{1}{|c}{\labbox{5}} & 
		\multicolumn{1}{|c|}{0}  
		\\\cline{5-6}
	& & & &&
		\multicolumn{1}{|c|}{\labbox{6}} 
		\\\cline{6-6}
	\end{array} .%
%		\begin{array}{cccccc}
%			\fbox{1}&R&R&0&R&0\\
%			&\fbox{2}&R&S&0&0\\
%			&&\fbox{3}&0&S&S\\
%			&&&\fbox{4}&R&S\\
%			&&&&\fbox{5}&0\\
%			&&&&&\fbox{6}.
%		\end{array}	
\end{equation*}
%Again, the two ``$R$'''s in positions $2,3$ and $4,5$ are blocked ($V1_{236}$, $V1_{245}$) so the only allowed division is 
%again by \taus{1}, resulting in
%\begin{equation*}
%	\SRT{V2=\ldiv{\taus{1}}{V1}}=
%	\begin{array}{cccccc}
%		\hline
%		\multicolumn{1}{|c}{\labbox{1}} & 
%		\multicolumn{1}{|c}{0} & 
%		\multicolumn{1}{|c}{R} & 
%		\multicolumn{1}{|c}{S} &
%		\multicolumn{1}{|c}{0} &
%		\multicolumn{1}{|c|}{0}
%		\\\hline 
%	& 
%		\multicolumn{1}{|c}{\labbox{2}} & 
%		\multicolumn{1}{|c}{R} & 
%		\multicolumn{1}{|c}{0} &
%		\multicolumn{1}{|c}{R} &
%		\multicolumn{1}{|c|}{0} 
%		\\\cline{2-6}
%	& &
%		\multicolumn{1}{|c}{\labbox{3}} & 
%		\multicolumn{1}{|c}{0} & 
%		\multicolumn{1}{|c}{S} &
%		\multicolumn{1}{|c|}{S} 
%		\\\cline{3-6}
%	& &&
%		\multicolumn{1}{|c}{\labbox{4}} & 
%		\multicolumn{1}{|c}{R} & 
%		\multicolumn{1}{|c|}{S} 
%		\\\cline{4-6}
%	& &&&
%		\multicolumn{1}{|c}{\labbox{5}} & 
%		\multicolumn{1}{|c|}{0}  
%		\\\cline{5-6}
%	& & & &&
%		\multicolumn{1}{|c|}{\labbox{6}} 
%		\\\cline{6-6}
%	\end{array}%
%\end{equation*}
which is totally blocked (the ``S'' in position $2,4$ blocks the ``$R$'''s in positions $1,2$ and $4,5$, 
while the ``$S$'' in position $3,5$ blocks the ``$R$'' in position $2,3$).  
One might say that the original $V$ (as well as $V1$) is \deffont{virtually totally blocked}--every sequence 
of allowed left divisions by transpositions eventually leads to a totally blocked matrix. Such a matrix clearly cannot represent a positive braid.

%In principle, there could be virtually blocked matrices which take more than one division to become totally blocked.  
%However, we have not succeeded in constructing such an example.  
%Consider, for example, the matrix {equation*}
%\begin{equation*}
%	W1=\taus{2}\bmult V1=
%	\left(\begin{array}{cccccc}
%		0 & 1 & 1 & 1 & 0 & 0 \\
%		0 & 0 & 1 & 0 & 1 & 1 \\
%		1 & 1 & 0 & 0 & 1 & 0 \\
%		1 & 0 & 0 & 0 & 1 & 1 \\
%		0 & 1 & 0 & 0 & 0 & 0 \\
%		0 & 1 & 0 & 1 & 0 & 0
%	\end{array}\right)
%\end{equation*}
%obtained by multiplying $V$ by \taus{2} instead of \taus{1}
%\begin{equation*}
%	W2=\taus{4}\bmult\taus{2}\bmult V=\taus{4}\bmult W1=
%	\left(\begin{array}{cccccc}
%		0 & 1 & 1 & 0 & 1 & 0 \\
%		0 & 0 & 1 & 1 & 0 & 1 \\
%		1 & 1 & 0 & 1 & 0 & 0 \\
%		0 & 1 & 0 & 0 & 1 & 0 \\
%		1 & 0 & 0 & 1 & 0 & 1 \\
%		0 & 1 & 0 & 0 & 1 & 0
%	\end{array}\right)
%\end{equation*}
%
%
This leads to a brute-force scheme for determining whether 
a given matrix $A\in\SRp{N}$ is the crossing matrix of some positive braid: 
quite simply, one tries every allowable sequence of left divisions by transpositions until one reaches either the zero matrix (in which
case the original matrix has been factored into $R$-matrices, which exhibits a realization of $A$) or one reaches a totally 
blocked matrix.  Since left division by a transposition subordinate to $A\in\SRp{N}$ 
reduces $N(A)$, the sum of the entries of $A$, by $1$, 
any string of allowable left divisions will terminate in one of these two possibilities after at most  $N(A)$ steps.  Of course the 
number of allowable sequences of left divisions is \textit{a priori} on the order of $N(A)!$, but we have implemented this scheme
in Mathematica code which can handle matrices of size up to about \sqsize{7} on a MacBook Pro with 8GB of memory.%
\footnote{We record our deep thanks to our colleague Bruce Boghosian, who spent many hours helping us develop this code,
as well as our former student Dan Fortunato, who helped us at the beginning of this process.}

Of course, this scheme provides an algorithmic characterization of crossing matrices for positive braids:
\begin{theorem}\label{thm:virtual}
	The crossing matrix of any positive braid on $N$ strands belongs to \SRp{N}: it is a non-negative integer matrix with 
	an $SR$ decompostion.  
	
	Conversely, every such matrix $A\in\SR{N}$ is either the crossing matrix of some (perhaps 
	many) positive braid on $N$ strands, or it is virtually totally blocked (\textit{i.e.,} every sequence of left divisions by
	transpositions corresponding to positions in the first superdiagonals of successive left quotients terminates in a totally
	blocked matrix).
\end{theorem}
However, this fails to be the kind of conceptual characterization of such crossing matrices which  we would like to see: 
a criterion which
can be applied directly to a matrix without an exhaustive search through allowable sequences of left divisions.  We have not
succeeded in formulating even a conjectural version of such a criterion.

\section{The realizability problem for symmetric matrices}\label{sec:positive}
The fact that the crossing product operation \bmult{} is simply matrix addition when restricted to \SymMat{N} 
makes it easier to think about realizability in this case--for example, 
\begin{remark}\label{rmk:sum}
	A sum of  realizable symmetric matrices is automatically
	realizable as a composition of the individual realizations. 
\end{remark} 

Also, we note that for any $R$-matrix $R$ (which by the discussion in Subsection~\ref{subsec:OR} is the crossing
matrix  \Rs{\pi} of some permutation $\pi\in \SymGp{N}$) the symmetrization \Ssub{\pi} of $R$ 
is realized by $\pbraid{\pi}\pbraid{(\invbar{\pi})}$.  Combined with Remark~\ref{rmk:sum} this shows
\begin{remark}\label{rmk:T1}
	If $S\in\SymMatp{N}$ is $T0$ and $T1$, it is (positively) realizable.
\end{remark} 

Using these observations together with a case-by-case argument, we can show that for $N=4$,
any $T0$ matrix in \SymMatp{N} is (positively) realizable;  
\begin{theorem}\label{thm:T04}
	A symmetric \sqsize{4} non-negative integer matrix with zero diagonal is positively realizable if and only if it is $T0$.
\end{theorem}
%We will denote the set of $T0$ matrices in \SymMatp{N} \resp{\SymMatO{N}} by replacing the superscript ``$0$''
%with ``$T0$'': \SymMatpT{N} \resp{\SymMatOT{N}}.  
It is natural, based on this and our experimental evidence using the algorithm resulting from Theorem~\ref{thm:virtual}, that Theorem~\ref{thm:T04} extends to all $N$. 

An approach to trying to prove the analogue of Theorem~\ref{thm:T04} for all $N$ might be via an induction on the number of 
nonzero entries in the matrix, using Theorem~\ref{thm:T04} to establish an initial case, and then the following idea,
which we find plausible but have not succeeded in proving or disproving.  We call an entry \as{ij} of the $T0$ matrix $A$
\deffont{fully supported} (even if that entry is zero) if, for every $k$ between $i$ and $j$, 
at least one of the entries \as{ik} and \as{kj}
is nonzero.

\begin{conjecture}\label{conj:fullsupport}
	Suppose $A$ is a positively realizable symmetric matrix, and $\as{ij}=0$ but the position is fully supported in $A$.
	Then there is some realization of $A$ in which strands $i$ and $j$ become adjacent somewhere, so that changing
	\as{ij} from zero to one results in a positively realizable matrix. 
\end{conjecture}

We caution that the word ``some'' is necessary here: the tableau
\begin{equation*}
	\fourTableau{S}{0}{0}{S}{0}{S}\quad=
	\fourTableau{S}{0}{0}{0}{0}{0}\quad\oplus
	\fourTableau{0}{0}{0}{S}{0}{0}\quad\oplus
	\fourTableau{0}{0}{0}{0}{0}{S}
\end{equation*}
can be realized as a composition of three ``hooks'' in any order.  
If the middle one in the sum above occurs before or after 
both of the others (Figure~\ref{fig:MA}), strands $1$ and $4$ come into adjacent positions, 
\begin{figure}[p]
\begin{center}
\begin{tikzpicture}
	\braid[
	line width=2pt,
	style strands={1}{red},
	style strands={2}{blue},
	style strands={3}{green},
	style strands={4}{brown}
	]
	(braid) at (2,0) s_{1}-s_{3} s_{1}-s_{3} s_{2} s_{2} ;

	\node[at=(braid-1-s),pin=north : 1] {};
	\node[at=(braid-2-s),pin=north : 2] {};
	\node[at=(braid-3-s),pin=north : 3] {};
	\node[at=(braid-4-s),pin=north : 4] {};
%	\node[at=(braid-5-s),pin=north : 5] {};

	\node[at=(braid-1-e),pin=south : 1] {};
	\node[at=(braid-2-e),pin=south : 2] {};
	\node[at=(braid-3-e),pin=south : 3] {};
	\node[at=(braid-4-e),pin=south : 4] {};
%	\node[at=(braid-5-e),pin=south : 5] {};

\end{tikzpicture}
\caption{Middle hook after the others: strands 1 and 4 adjacent}
\label{fig:MA}
\end{center}
\end{figure}
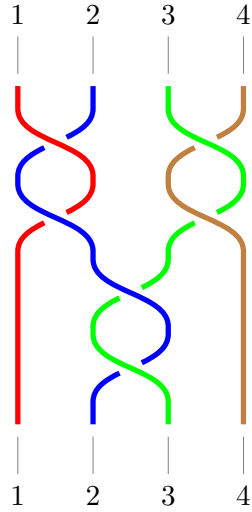 
but if it occurs between them (Figure~\ref{fig:MB}), then at any stage 
either strand 2 or strand 3 separates these two. 
\begin{figure}[p]
\begin{center}
\begin{tikzpicture}
	\braid[
	line width=2pt,
	style strands={1}{red},
	style strands={2}{blue},
	style strands={3}{green},
	style strands={4}{brown}
	]
	(braid) at (2,0) s_{1} s_{1} s_{2}  s_{2}  s_{3} s_{3} ;

	\node[at=(braid-1-s),pin=north : 1] {};
	\node[at=(braid-2-s),pin=north : 2] {};
	\node[at=(braid-3-s),pin=north : 3] {};
	\node[at=(braid-4-s),pin=north : 4] {};
%	\node[at=(braid-5-s),pin=north : 5] {};

	\node[at=(braid-1-e),pin=south : 1] {};
	\node[at=(braid-2-e),pin=south : 2] {};
	\node[at=(braid-3-e),pin=south : 3] {};
	\node[at=(braid-4-e),pin=south : 4] {};
%	\node[at=(braid-5-e),pin=south : 5] {};

\end{tikzpicture}
\caption{Middle between the others: strands 1 and 4 separated by strands 2 and 3}
\label{fig:MB}
\end{center}
\end{figure}
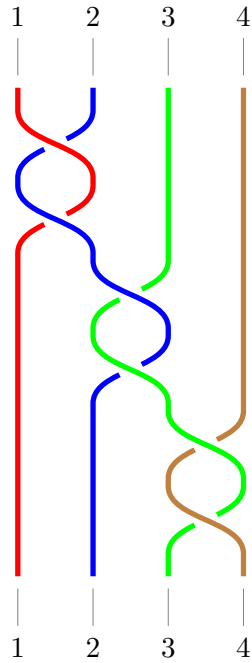 

Given this conjecture as a lemma, an inductive argument goes as follows: in any $T0$ matrix, there are nonzero entries which
do not support any other element (since all supports of an entry live in lower-numbered superdiagonals).  ``Erasing'' one such entry
yields a $T0$ matrix with a lower entry sum;  by induction on this sum, the latter is realizable, and hence by the conjectured 
lemma, so is our given matrix.

\bibliography{braids}

\providecommand{\bysame}{\leavevmode\hbox to3em{\hrulefill}\thinspace}
\providecommand{\MR}{\relax\ifhmode\unskip\space\fi MR }
\providecommand{\MRhref}[2]{%
  \href{http://www.ams.org/mathscinet-getitem?mr=#1}{#2}
}
\providecommand{\href}[2]{#2}
\begin{thebibliography}{1}

\bibitem{Artin1}
E.~Artin, \emph{\foreignlanguage{german}{Theorie der {Z}{\"o}pfe}},
  \foreignlanguage{german}{Abhandlungen der Mathematisches Seminar
  Universit{\"a}t Hamburg} \textbf{4} (1925), no.~1, 47--72.

\bibitem{Artin2}
\bysame, \emph{Theory of braids}, Annals of Mathematics \textbf{48} (1947),
  no.~2, 101--126.

\bibitem{BGKN}
J.~Burillo, M.~Gutierrez, S.~Krsti{\'c}, and Z.~Nitecki, \emph{Crossing
  matrices and {T}hurston's normal form for braids}, Topology and its
  Applications \textbf{118} (2002), 293--308.

\bibitem{Thurston}
W.~Thurston, \emph{Braid groups}, Word processing in groups (D.~Epstein, ed.),
  Jones and Bartlett, Boston, MA, 1992, pp.~181--209.

\end{thebibliography}
\bibliographystyle{amsplain}

%\end{bibliography}
\end{document}